\documentclass{amsart}
\usepackage{amssymb}
\usepackage{eucal}
\usepackage{amssymb,latexsym}
\theoremstyle{plain}
\newtheorem{theorem}{Theorem}

\newtheorem{lemma}{Lemma}

\theoremstyle{definition}

\theoremstyle{remark}

\newtheorem{remark}{Remark}
\numberwithin{equation}{section}
\newcommand{\e}{\epsilon}

\newcommand{\R}{\mathbb R}

\newcommand{\Rn}{\mathbb R^n}

\newcommand{\Rm}{\mathbb R^{n+1}}
\begin{document}

\title[Hardy Uncertainty Principle, Convexity and Parabolic Evolutions]{Hardy Uncertainty Principle, Convexity and Parabolic Evolutions}
%%%%%%%%%%%%%%%%%%
%Author information
%%%%%%%%%%%%%%%%%%%%
\author{L. Escauriaza}
\address[L. Escauriaza]{UPV/EHU\\Dpto. de Matem\'aticas\\Apto. 644, 48080 Bilbao, Spain.}
\email{luis.escauriaza@ehu.eus}
\thanks{The first and fourth authors are supported  by the grants MTM2014-53145-P and IT641-13 (GIC12/96), the second author by NSF grant DMS-0456583, and the fourth author is also supported by SEV-2013-0323.}

%\thanks{}
%%%%%%%%%%%%%%%%%%%%%%
\author{C. E. Kenig}
\address[C. E. Kenig]{Department of Mathematics\\University of Chicago\\Chicago, Il. 60637 \\USA.}
\email{cek@math.uchicago.edu}
%\thanks{}
%%%%%%%%%%%%%%%%%%%%%%
\author{G. Ponce}
\address[G. Ponce]{Department of Mathematics\\
University of California\\
Santa Barbara, CA 93106\\
USA.}
\email{ponce@math.ucsb.edu}
%\thanks{}
%%%%%%%%%%%%%%%%%%%%%%
\author{L. Vega}
\address[L. Vega]{UPV/EHU\\Dpto. de Matem\'aticas\\Apto. 644, 48080 Bilbao, Spain.}
\email{luis.vega@ehu.eus}
%\thanks{}
%%%%%%%%%%%%%%%%%%%%%%
%\date{}
%\dedicatory{}
%%%%%%%%%%%%%%
\begin{abstract}
We give a new proof of the $L^2$ version of Hardy's uncertainty principle based on calculus and on its dynamical version for the heat equation. The reasonings rely on new log-convexity properties and the derivation of optimal Gaussian decay bounds for solutions to the heat equation with Gaussian decay at a future time. We extend the result to heat equations with lower order variable coefficient.
\end{abstract}
\maketitle
%%%%%%%%%%%%%%%%%%
\begin{section}{Introduction}\label{S: Introduction}
In this paper we continue the study in \cite{kpv02,ekpv06,ekpv08a,ekpv08b,ekpv09, cekpv10} related to the Hardy uncertainty principle and its relation to unique continuation properties for some evolutions.

One of our motivations  came from a well known result due to G. H. Hardy (\cite{Hardy}, \cite[pp. 131]{StSh}), which concerns the decay of a function $f$ and its Fourier transform, 
\[\widehat f(\xi)=(2\pi)^{-\frac n2}\int_{\Rn}e^{-i\xi\cdot x}f(x)\,dx.\]

\emph{If 
$f(x)=O(e^{-|x|^2/\beta^2})$, $\widehat f(\xi)=O(e^{-4|\xi|^2/\alpha^2})$ and  $1/\alpha\beta>1/4$, then $f\equiv 0$. Also, if $1/\alpha\beta=1/4$, $f$ is a constant multiple of $e^{-|x|^2/\beta^2}$.}

\vspace{0,1 cm}
As far as we know, the known proofs for this result and its variants - before the one in \cite{kpv02,ekpv06,ekpv08b,ekpv09, cekpv10} - use complex analysis (the Phragm\'en-Lindel\"of principle). There has also been considerable interest in a better understanding of this result and on extensions of it to other settings: \cite{CoPr}, \cite{Hor2}, \cite{SiSu}, \cite{bonamie} and \cite{bonamie2}.

The result can be rewritten in terms of the free solution of the Schr\"odinger equation 
\begin{equation*}
i\partial_tu+\triangle u=0,\ \text{in}\ \Rn\times (0,+\infty),
\end{equation*}
with initial data $f$,
\begin{equation*}
u(x,t)= (4\pi it)^{-\frac n 2} \int_{\Rn}e^{\frac{ i |x-y|^2}{4t}}f(y)\, dy= \left(2\pi it\right)^{-\frac n2}e^{\frac{i|x|^2}{4t}}\widehat{e^{\frac{i|\,\cdot\,|^2}{4t}}f}\left(\frac x{2t}\right),
\end{equation*}
in the following way: 
\vspace{0,1 cm}

\emph{If $u(x,0)=O(e^{-|x|^2/\beta^2})$, $u(x,T)=O(e^{-|x|^2/\alpha^2})$ and $T/\alpha\beta> 1/4$, then $u\equiv 0$. Also, if $T/\alpha\beta=1/4$, $u$ has as initial data a constant multiple of $e^{-\left(1/\beta^2+i/4T\right)|y|^2}$.}

The corresponding results in terms of $L^2$-norms, established in \cite{CoPr2}, are the following: 
\vspace{0,1 cm}

\emph{If $e^{|x|^2/\beta^2}f$, $e^{4|\xi |^2/\alpha^2}\widehat f$ are in $L^2(\Rn)$ and $1/\alpha\beta\ge 1/ 4$, then $f\equiv 0$.}

\vspace{0,1 cm}
\emph{If $e^{|x|^2/\beta^2}u(x,0)$, $e^{|x|^2/\alpha^2}u(x,T)$ are in $L^2(\Rn)$ and $T/\alpha\beta\ge 1/4$, then $u\equiv 0$.}

In \cite {{ekpv09}} we proved a uniqueness result in this direction for variable coefficients Schr\"odinger evolutions
\begin{equation}\label{E: 1.1}
\partial_t u=i\left(\triangle u+V(x,t) u\right)\ ,\ \text{in}\  \R^n\times [0,T].
\end{equation}
with bounded potentials $V$ verifying, $V(x,t)=V_1(x)+V_2(x,t)$, with $V_1$ real-valued and
\[\sup_{[0,T]}\|e^{T^2|x|^2/\left(\alpha t+\beta\left(T-t\right)\right)^2}V_2(t)\|_{L^\infty(\Rn)}<+\infty\]
or
\begin{equation*}
\lim_{R\rightarrow +\infty}\int_0^T\|V(t)\|_{L^\infty(\Rn\setminus B_R)}\,dt =0.
\end{equation*}
More precisely, we showed that the only solution $u$ to \eqref{E: 1.1} in $C([0,T], L^2(\Rn))$, which verifies 
\begin{equation*}
\|e^{|x|^2/\beta^2}u(0)\|_{L^2(\Rn)}+\|e^{|x|^2/\alpha^2}u(T)\|_{L^2(\Rn)}<+\infty
\end{equation*}
is the zero solution, when $T/\alpha\beta>1/ 4$. When $T/\alpha\beta=1/4$, we found a complex valued potential potential $V$ with 
\begin{equation*}
|V(x,t)|\lesssim\frac 1{1+|x|^2},\ \text{in}\ \R^n\times [0,T]
\end{equation*}
 and a nonzero smooth solution $u$ in $C^\infty([0,T],\mathcal S(\Rn))$ of \eqref{E: 1.1} with
 \begin{equation*}
 \|e^{|x|^2/\beta^2}u(0)\|_{L^2(\Rn)}+\|e^{|x|^2/\alpha^2}u(T)\|_{L^2(\Rn)}< +\infty.
 \end{equation*}

Thus, we established in \cite{ekpv09} that the  optimal version of Hardy's Uncertainty Principle in terms of $L^2$-norms holds for solutions to \eqref{E: 1.1} holds when $T/\alpha\beta>1/4$ for many general bounded potentials, while it can fail for some complex-valued potentials in the end-point case, $T/\alpha\beta=1/4$. Finally, in \cite{cekpv10} we showed that the reasonings in \cite{kpv02,ekpv06,ekpv08a,ekpv08b,ekpv09, cekpv10} provide the first proof (up to the end-point case) that we know of Hardy's uncertainty principle for the Fourier transform without the use of holomorphic functions.

The Hardy uncertainty principle also has a dynamical version associated to the heat equation,
\begin{equation*}
\partial_tu-\Delta u=0,\ \text{in}\ \R^n\times (0,+\infty),
\end{equation*}
with initial data $f$,
\begin{equation*}
u(x,t)= (4\pi t)^{-n/2} \int_{\Rn}e^{- |x-y|^2/4t}f(y)\, dy,\ \widehat{u}(\xi,t)= e^{-t|\xi|^2}\widehat{f}(\xi),\ x,\ \xi\in\Rn,\ t>0.
\end{equation*}
\vspace{0,1 cm}
In particular, its $L^\infty$ and $L^2$ versions yield the following statements:
\vspace{0,1 cm}

\emph{
If $u(0)$ is a finite measure in $\Rn$, $u(x,T)=O(e^{-|x|^2/\delta^2})$ and  $\delta<\sqrt{4T}$, then $f\equiv 0$. Also, if $\delta=\sqrt{4T}$, then $u(0)$ is a multiple of the Dirac delta function.}
\vspace{0,1 cm}

\emph{
If $u(0)$ is in $L^2(\Rn)$, $\|e^{|x|^2/\delta^2}u(T)\|_{L^2(\Rn)}$ is finite and $\delta\le\sqrt{4T}$, then $u\equiv 0$.}
\vspace{0,1 cm}

In \cite[Theorem 4]{ekpv08b} we proved that a dynamical $L^2$-version of Hardy uncertainty principle holds for solutions $u$ in $C([0,T], L^2(\Rn))\cap L^2([0,T], H^1(\Rn))$ to
\begin{equation}\label{E: 1.2}
\partial_tu=\Delta u+V(x,t)u,\ \text{in}\ \Rn\times [0,T],
\end{equation}
when $V$ is any bounded complex potential in $\Rn\times [0,T]$ and $\delta <\sqrt{T}$. Here, we find the optimal interior Gaussian decay over $[0,1]$ for solutions to \eqref{E: 1.2} with
\begin{equation*}
\|e^{|x|^2/\delta^2}u(T)\|_{L^2(\Rn)}<+\infty,
\end{equation*}
when $\delta>\sqrt{4T}$ and derive from it the full dynamical $L^2$ version of the Hardy uncertainty principle for solutions to \eqref{E: 1.2}, reaching the end-point case, $\delta=\sqrt{4T}$.

\begin{theorem}\label{T: lamejora1}
Assume that $u$ in $C([0,T], L^2(\Rn))\cap L^2([0,T], H^1(\Rn))$ verifies \eqref{E: 1.2} with $V$ in $L^\infty(\Rn\times [0,T])$. Assume that 
\begin{equation}\label{E: 0}
\|e^{T|x|^2/4(T^2+R^2)}u(T)\|_{L^2(\Rn)}<+\infty
\end{equation}
for some $R>0$. Then, there is a universal constant $N$ such that
\begin{multline}\label{E: loque hay que conseguir}
\sup_{[0,T]}\|e^{t|x|^2/4(t^2+R^2)}u(t)\|_{L^2(\Rn)}\\
\le e^{N\left(1+T^2\|V\|^2_{L^\infty(\Rn\times [0,T])}\right)}\left[\|u(0)\|_{L^2(\Rn)}+\|e^{T|x|^2/4(T^2+R^2)}u(T)\|_{L^2(\Rn)}\right].
\end{multline}
Moreover, $u$ must be identically zero when $\|e^{|x|^2/4T}u(T)\|_{L^2(\Rn)}$ is finite.
\end{theorem}
Theorem \ref{T: lamejora1} is optimal because
\begin{equation}\label{E: el enemigo}
u_R(x,t)=(t-iR)^{-\frac n2}e^{-|x|^2/4(t-iR)}=(t-iR)^{-\frac n2}e^{-(t+iR)|x|^2/4(t^2+R^2)},
\end{equation}
is a solution to the heat equation and for each fixed $t>0$, $t/4(t^2+R^2)$ is decreasing in the $R$-variable for $R>0$ . Also, observe that $t/4(t^2+R^2)$ attains its maximum value in the interior of $[0,T]$, when $R\neq T$, 

Notice that the finiteness condition on condition on $\|e^{|x|^2/4T}u(T)\|_{L^2(\Rn)}$ is independent of the size of the potential or the dimension and that we do not assume any regularity or strong decay of the potentials.  

This improvement of our results in \cite[Theorem 4]{ekpv08b} on the relation between Hardy uncertainty principle and its dynamical version for parabolic evolutions comes from  a better understanding of the solutions to \eqref{E: 1.2} which have Gaussian decay and of the adaptation to the parabolic context of the same kind of log-convexity arguments that we used in \cite{ekpv09} to derive the dynamical version of the Hardy uncertainty principle for Schr\"odinger evolutions.

We have not tried to  extend the results in Theorems \ref{T: lamejora1} to parabolic evolutions with nonzero drift terms
\begin{equation}\label{E: maÁgnetic}
\partial_tu=\Delta u+W(x,t)\cdot\nabla u+V(x,t)u.
\end{equation}
We expect that similar methods will yield  analogue results for solutions to \eqref{E: maÁgnetic} (See \cite{ds} for initial results following the approach initiated in \cite{kpv02} and \cite{ekpv06} for the case of Sch\"rodiger evolutions).

In what follows, $N$ denotes a universal constant depending at most on the dimension, $N_{a,\xi,\dots}$ a constant depending on the parameters $a,\xi,\dots$ In section \ref{S: a few lemmas} we give three Lemmas which are necessary for our proof  in section  \ref{S:2} of Theorem \ref{T: lamejora1}.
 \end{section}
\begin{section}{A few Lemmas}\label{S: a few lemmas}
In the sequel \[\left(f,g\right)=\int_{\Rn}f\overline g\,dx\ ,\ \|f\|^2=\left(f,f\right)\ \text{and}\ \|V\|_\infty=\|V\|_{L^\infty(\Rn\times [0,1])}.\]

In Lemma \ref{L: aproximaci—n de convexidad logar'tmica}, $\mathcal S$ and $\mathcal A$ denote respectively a symmetric and skew-symmetric bounded linear operators on $\mathcal S(\Rn)$. Both are allowed to depend smoothly on the time-variable, $\mathcal S_t = \partial_t\mathcal S$ and $[\mathcal S,\mathcal A]$ is the space commutator of $\mathcal S$ and $\mathcal A$. The reader can find a proof of Lemma \ref{L: aproximaci—n de convexidad logar'tmica} in \cite[Lemma 2]{ekpv09}. 
\begin{lemma}\label{L: aproximaci—n de convexidad logar'tmica} Let $\mathcal S$ and $\mathcal A$ be as above, $f$ lie in $C^\infty([c,d],\mathcal S(\Rn))$ and $\gamma: [c,d]\longrightarrow (0,+\infty)$ be a smooth function such that
\begin{equation*}
\left(\gamma\,\mathcal S_t f(t)+\gamma\left[\mathcal S,\mathcal A\right]f(t)+\dot\gamma\,\mathcal Sf(t),f(t)\right)\ge 0, \ \text{when}\ c\le t\le d.
\end{equation*}
Then, if $H(t)=\|f(t)\|^2$ and $\epsilon >0$
\begin{equation*}
H(t)+\epsilon\le \left(H(c)+\epsilon\right)^{\theta(t)}\left(H(d)+\epsilon\right)^{1-\theta(t)}e^{M_\epsilon(t)+2N_\e(t)}, \ \text{when}\ c\le t\le d,
\end{equation*}
where $M_\e$ verifies
\begin{equation*}
\partial_t\left(\gamma\,\partial_tM_\e\right)=-\gamma\,\frac{\|\partial_tf-\mathcal Sf-\mathcal Af\|^2}{H+\e}\, ,\ \text{in}\ [c,d],\quad M_\e(c)=M_\e(d)=0,
\end{equation*}
\begin{equation*}
N_\e=\int_c^d\left |\text{\it Re}\ \frac{\left(\partial_sf(s)-\mathcal Sf(s)-\mathcal Af(s),f(s)\right)}{H(s)+\e}\right |\,ds
\end{equation*}
and 
\begin{equation*}\label{E: el jodidoexponente}
\theta(t)=\frac{{\int_t^d}\frac{ds}{\gamma}}{\int_c^d\frac{ds}{\gamma}}\, .
\end{equation*}
\end{lemma}
 A calculation (see formulae (2.12), (2.13) and (2.14) in \cite{ekpv08b} with $\gamma=1$) shows that given smooth functions $a:[0,1]\longrightarrow [0,+\infty)$, $b: [0,1]\longrightarrow\R$ and $T:[0,1]\longrightarrow\R$, and $\xi$ in $\Rn$
\begin{equation*}
e^{a(t)|x|^2+ b(t)x\cdot\xi-T(t)|\xi|^2}\left(\partial_t-\triangle\right)e^{-a(t)|x|^2- b(t)x\cdot\xi+T(t)|\xi|^2}=\partial_t-\mathcal S-\mathcal A,
\end{equation*}
where $\mathcal S$ and $\mathcal A$ are the symmetric and skew-symmetric linear bounded operators on $\mathcal S(\Rn)$ given by
\begin{align}
&\mathcal S=\Delta+\left(a'+4a^2\right)|x|^2+\left(b'+4ab\right)x\cdot \xi+\left(b^2-T'\right)|\xi|^2,\label{E: parte semietrica}\\
&\mathcal A=-2\left(2ax+b\,\xi\right)\cdot\nabla -2na.\label{E: parteantisimetrica}
\end{align}
and
\begin{multline}\label{E: el conmutador}
\mathcal S_t+\left[\mathcal S,\mathcal A\right]=-8a\,\Delta +\left(a''+16aa'+32a^3\right)|x|^2\\+\left(b''+8ab'+8a'b+32a^2b\right) x\cdot\xi+\left(8ab^2+4bb'-T''\right)|\xi|^2.
\end{multline}

In Lemma \ref{L: el calculo del conmutador mas largo} we make choices of $a$, $b$ and $T$ which make non-negative the self-adjoint operator 
\begin{equation*}
e^{8A}\left(\mathcal S_t+\left[\mathcal S,\mathcal A\right]\right)+\left(e^{8A}\right)'\mathcal S,
\end{equation*}
where $A$ denotes an anti-derivative of $a$ in $[0,1]$ with $A(1)=0$, .

\begin{lemma}\label{L: el calculo del conmutador mas largo} Let $a:[0,1]\longrightarrow\R$ be a smooth function verifying 
\begin{equation}\label{E: condicionnecesa}
\left(e^{8A}a\right)''\ge 0,\ \text{in}\ [0,1],
\end{equation}
and let $b$ and $T$ be the solutions to
\begin{equation}\label{E: la formula que define}
\begin{cases}
\left(e^{8A}b \right)''=2\left(e^{8A}a \right)'',\ \text{in}\ [0,1],\\
b(0)=b(1)=0,
\end{cases}
\end{equation}
and
\begin{equation}\label{E: la formula que define2}
\begin{cases}
\left(e^{8A}T'\right)'=2\left(e^{8A}b^2\right)'-\left(e^{8A}a\right)'',\ \text{in}\ [0,1],\\
 T(0)=T(1)=0.
\end{cases}
\end{equation}
Then, 
\begin{equation*}
\left(e^{8A}\mathcal S_tf+e^{8A}\left[\mathcal S,\mathcal A\right]f+\left(e^{8A}\right)'\mathcal Sf,f\right)\ge 0,\ \text{when}\ f\in \mathcal S(\Rn)\ \text{and}\ 0\le t\le 1.
\end{equation*}
\end{lemma}
\begin{proof} From \eqref{E: parte semietrica}, \eqref{E: el conmutador}, the identities
\begin{align}
&\left(e^{8A}a\right)''=e^{8A}\left(a''+24aa'+ 64a^3\right).\notag\\
&\left(e^{8A}b\right)''=e^{8A}\left(b'' +16ab'+8a' b+64a^2b\right).\notag\\
&\left(e^{8A}b^2\right)'=e^{8A}\left(8ab^2+2bb'\right),\notag
\end{align}
and the definitions of $b$ and $T$, we have
\begin{multline*}  
e^{8A}\left(\mathcal S_t+\left[\mathcal S,\mathcal A\right]\right)+\left(e^{8A}\right)' \mathcal S\\
= \left(e^{8A}a\right)'' |x|^2+\left(e^{8A}b\right)''x\cdot \xi+\left(2\left(e^{8A}b^2\right)'-\left(e^{8A}T'\right)'\right)|\xi|^2 \\
=\left(e^{8A}a\right)'' \left(|x|^2+2 x\cdot \xi+|\xi|^2\right)=\left(e^{8A}a\right)'' |x+\xi|^2.\\
\end{multline*}
The later and  \eqref{E: condicionnecesa} implies Lemma \ref{L: el calculo del conmutador mas largo}.
\end{proof}
In the next Lemma we assume that $u$ in $C([0,1], L^2(\Rn))\cap L^2([0,1], H^1(\Rn))$ verifies \eqref{E: 1.2} in $\Rn\times (0,1]$ and  
\begin{equation*}
\|e^{|x|^2/\delta^2}u(1)\| <+\infty.
\end{equation*}
\begin{lemma}\label{L: el calculo del conmutador mas largo2} Let $a:[0,1]\longrightarrow [0,+\infty)$ be a smooth function with $a(0)=0$, $a(1)=1/\delta^2$, $\left(e^{8A}a\right)''>0$ in $[0,1]$ and
\begin{equation*}
\sup_{[0,1]}\| e^{\left(a(t)-\e\right)|x|^2}u(t)\|< +\infty,\ \text{when}\ 0<\e\le 1.
\end{equation*}
Then, there is a universal constant $N$ such that for $b$ and $T$ as in \eqref{E: la formula que define} and \eqref{E: la formula que define2},
\begin{equation*}
\|e^{a(t)|x|^2+b(t)x\cdot\xi-T(t)|\xi|^2}u(t)\|\le e^{N\left(1+\|V\|_{\infty}^2\right)}\left(\|u(0)\|+\|e^{|x|^2/\delta^2}u(1)\|\right),
\end{equation*}
when $\xi$ is in $\Rn$ and $0\le t\le 1$.
\end{lemma}
\begin{proof}
 For $\xi$ in $\Rn$ and $\e>0$, set 
\begin{equation*}
f_\e(x,t)=e^{a_\e|x|^2+b_\e x\cdot \xi-T_\e |\xi|^2}u(x,t),
\end{equation*}
with $a_\e=a-\e$, $A_\e=A+\e(1-t)$, and with $b_\e$ and $T_\e$ as in Lemma \ref{L: el calculo del conmutador mas largo} but with $a$ and $A$ replaced by $a_\e$ and $A_\e$ respectively.
The local Schauder estimates for solutions to \eqref{E: 1.2} show that 
\begin{multline*}
r|\nabla u(x,t)|+ r^2\left(\text{\rlap |{$\int_{B_r(x)\times (t-r^2,t]}$}}|\partial_su|^p+|D^2u|^p\, dyds\right)^{\frac{1}{p}}\\
\le N_{p}\left(1+r^2\|V\|_{L^\infty(\Rn\times [0,1])}\right)\text{\rlap |{$\int_{B_{2r}(x)\times (t-4r^2,t]}$}}|u|\, dyds
\end{multline*}
for $1< p<\infty$, $0<r\le \sqrt{t}/2$, $0<t\le 1$. Thus, $f_\e$ is in $W^{2,1}_2(\Rn\times[\varrho,1])$ and verifies
\begin{equation}\label{E: condicon de acotaci—n1}
\begin{split}
&\sup_{[0,1]}\|f_\e(t)\|\le N_{a,\e,\xi}\sup_{[0,1]}\|e^{\left(a-\frac\e2\right)|x|^2}u(t)\|,\\
&\sup_{[\varrho,1]}\|\nabla f_\e(t)\|\le N_{a,\e,\xi,\varrho}\sup_{[0,1]}\|e^{\left(a-\frac\e 2\right)|x|^2}u(t)\|\end{split}
\end{equation}
for $0<\varrho\le \frac12$ and
\begin{equation}\label{E: la ecuacion de f}
\partial_tf_\e-\mathcal S_\e f_\e-\mathcal A_\e f_\e=V(x,t)f_\e,\ \text{in}\ \Rn\times (0,1],
\end{equation}
where $\mathcal S_\e$ and $\mathcal A_\e$ are the operators defined in \eqref{E: parte semietrica} and \eqref{E: parteantisimetrica} with $a$, $A$, $b$ and $T$ replaced by $a_\e$, $A_\e$, $b_\e$ and $T_\e$ respectively. Also, \eqref{E: condicon de acotaci—n1}, the equation \eqref{E: la ecuacion de f} verified by $f_\e$  and \cite[Lemma 1.2]{Temam} show that $f_\e$ is in $C((0,1], L^2(\Rn))$.

Extend $f_\e$ as zero outside $\Rn\times [0,1]$ and let $\theta$ in $C^\infty(\Rm)$ be a mollifier supported in the unit ball of $\Rm$. For $0<\rho\le \frac 14$, set $f_{\e,\rho}=f_\e\ast\theta_\rho$ and
\begin{equation*}
\theta^{x,t}_\rho(y,s)=\rho^{-n-1}\theta(\tfrac{x-y}\rho\, ,\tfrac{t-s}\rho).
 \end{equation*}
 Then, $f_{\e,\rho}$ is in $C^\infty([0,1],\mathcal S(\Rn))$ and for $x$ in $\Rn$ and $\rho\le t\le 1-\rho$,
 \begin{equation}\label{E: formulalarga}
 \begin{split}
&\left(\partial_tf_{\e,\rho}-\mathcal S_\e f_{\e,\rho}-\mathcal A_\e f_{\e,\rho}\right)(x,t)=\left(Vf_{\e}\right)\ast\theta_\rho(x,t)\\
&+\int f_\e\left(q_\e(y,s,\xi)-q_\e(x,t,\xi)\right)\theta^{x,t}_\rho\,dyds\\
&+\int \nabla_yf_\e\cdot\left[\left(a_\e(t)x+2b_\e(t)\xi\right)-\left(a_\e(s)y+2b_\e(s)\xi\right)\right]\theta^{x,t}_\rho\,dyds,
\end{split}
\end{equation}
with
\begin{multline*}
q_{\e}(x,t,\xi)=\left(a_{\e}'(t)+4a_{\e}^2(t)\right)|x|^2\\+\left(b_{\e}'(t)+4a_{\e}(t)b_{\e}(t)\right)x\cdot\xi+\left(b_{\e}^2(t)-T_{\e}'(t)\right)|\xi|^2-2na_\e(t).
\end{multline*}
The last identity gives,
\begin{equation*}
\left(\partial_tf_{\e,\rho}-\mathcal S_\e f_{\e,\rho}-\mathcal A_\e f_{\e,\rho}\right)(x,t)=\left(Vf_\e\right)\ast\theta_\rho(x,t)+A_{\e,\rho}(x,t),
\end{equation*}
in $\Rn\times [\rho,1-\rho]$, where $A_{\e,\rho}$ denotes the sum of the second and third integrals in the right hand side of \eqref{E: formulalarga}. Moreover, from \eqref{E: condicon de acotaci—n1} there is $N_{a,\e,\xi,\varrho}$ such that for $0<\varrho<\frac 12$ and $0<\rho\le\varrho$,
\begin{equation*}
\sup_{[\varrho,1-\varrho]}\|A_{\e,\rho}(t)\|_{L^2(\Rn)}\le \rho N_{a,\e,\xi,\varrho}\sup_{[-1,1]}\|e^{\left(a(t)-\frac\e 2\right)|x|^2}u(t)\|. \label{E: controldeA}
\end{equation*}
Also, $\left(e^{8A_\e}a_\e\right)''>0$ in $[0,1]$, when $0<\e\le\e_a$,
and from Lemma \ref{L: el calculo del conmutador mas largo} we can apply to $\mathcal S_\e$, $\mathcal A_\e$ and $f_{\e,\rho}$, the conclusions of Lemma \ref{L: aproximaci—n de convexidad logar'tmica} with $[c,d]=[\varrho,1-\varrho]$, $\gamma=e^{8A_\e}$ and $H_{\e,\rho}(t)=\|f_{\e,\rho}(t)\|^2$. Thus,
\begin{equation}\label{E: casialli}
H_{\e,\rho}(t)\le \left(H_{\e,\rho}(\varrho)+H_{\e,\rho}(1-\varrho)+2\e\right)e^{M_{\e,\rho}(t)+2N_{\e,\rho}}, \ \text{when}\ \varrho\le t\le 1-\varrho,
\end{equation}
where $M_{\e,\rho}$ verifies
\begin{equation*}
\begin{cases}
\partial_t\left(e^{8A_\e}\partial_tM_{\e,\rho}\right)=-e^{8A_\e}\,\frac{\|\partial_tf_{\e,\rho}-\mathcal S_\e f_{\e,\rho}-\mathcal A_\e f_{\e,\rho}\|^2}{H_{\e,\rho}+\e}\, ,\ \text{in}\ [\varrho,1-\varrho],\\
M_{\e,\rho}(\varrho)=M_{\e,\rho}(1-\varrho)=0,
\end{cases}
\end{equation*}
and
\begin{equation*}
N_{\e,\rho}=\int_{\varrho}^{1-\varrho}\frac{\|\partial_sf_{\e,\rho}(s)-\mathcal S_\e f_{\e,\rho}(s)-\mathcal A_\e f_{\e,\rho}(s)\|}{\sqrt{H_{\e,\rho}(s)+\e}}\,ds.
\end{equation*}
We can now pass to the limit in \eqref{E: casialli}, when $\rho$ tends to zero and derive that for $H_\e(t)=\|f_\e(t)\|^2$, $0<\varrho\le \frac12$ and $0<\e\le\e_a$, we have
\begin{equation}\label{E: casialli2}
H_{\e}(t)\le \left[H_{\e}(\varrho)+H_{\e}(1-\varrho)+2\e\right]e^{M_{\e}(t)+2\|V\|_{\infty}}, \ \text{in}\ [\varrho,1-\varrho],
\end{equation}
with
\begin{equation}\label{E; uansolucion}
\begin{cases}
\partial_t\left(e^{8A_\e}\partial_tM_{\e}\right)=-\,e^{8A_\e}\,\frac{\|\partial_tf_{\e}-\mathcal S_\e f_{\e}-\mathcal A_\e f_{\e}\|^2}{H_{\e}}\, ,\ \text{in}\ [0,1].\\
M_{\e}(0)=M_{\e}(1)=0,
\end{cases}
\end{equation}

By writing an explicit formula for the solution to \eqref{E; uansolucion}, it follows from the monotonicity of $A$; i.e. $A'\ge 0$ in $[0,1]$ and \eqref {E: la ecuacion de f} that
\begin{equation*}
M_\e(t)\le N\left(1+\|V\|_{\infty}^2\right).
\end{equation*}

Also, there is $N_a>0$ such that $|b_{\e}'|+|T_{\e}'|\le N_a$, when $0<\e\le \e_a$. The later, the continuity of $f_\e$ in $C((0,1], L^2(\Rn))$ and the fact that $a(0)=b_\e(0)=T_\e(0)=0$ show, that for each fixed $\xi\in\Rn$ and all $0<\e<\e_a$, there is $\varrho_\e$ with $\lim_{\e\to 0^+}\varrho_\e=0$ such that $H_{\e}(1-\varrho_\e)\le H_\e(1)+\e$ and $H_\e(\varrho_\e)\le \sup_{[0,1]}\|u(t)\|$. Thus, after taking $\varrho=\varrho_\e$ in \eqref{E: casialli2}, we get
\begin{equation*}
\|e^{a_\e (t)|x|^2+b_\e (t)x\cdot\xi-T_\e (t)|\xi|^2} u(t)\| 
\le  e^{N\left(1+\|V\|_{\infty}^2\right)}\left[\sup_{[0,1]} \|u(t)\|+\| e^{|x|^2/\delta^2}u(1)\|+3\e\right],
\end{equation*}
for $\varrho_\e\le t\le 1-\varrho_\e$. Then, let  $\e\to 0^+$ and recall the $L^2$ energy inequality verified by solutions to \eqref{E: 1.2}.
\end{proof}
\end{section}
\begin{section}{Proof of Theorem \ref{T: lamejora1}}\label{S:2}
\begin{proof} By scaling it suffices to prove Theorem \ref{T: lamejora1} when $T=1$. Assume first that $u$ in $C([0,1], L^2(\Rn))\cap L^2([0,1], H^1(\Rn))$ verifies \eqref{E: 1.2} in $\Rn\times (0,1]$ and  
\begin{equation*}
\|e^{|x|^2/\delta^2}u(1)\| <+\infty
\end{equation*}
for some $\delta>2$. Following \cite[Theorem 4]{ekpv08b}, for $\alpha=1$ and $\beta =1+\frac 2{\delta}$, define
\begin{equation*}
\widetilde u(x,t)=\left(\tfrac{\sqrt{\alpha\beta}}{\alpha (1-t)+\beta t}\right)^{\frac n2}u(\tfrac{\sqrt{\alpha\beta}x}{\alpha(1-t)+\beta t}, \tfrac{\beta t}{\alpha(1-t)+\beta t})e^{\frac{(\alpha-\beta)|x|^2}{4(\alpha(1-t)+\beta t)}}\,.
\end{equation*}
Then, $\widetilde u$ is in $C([0,1], L^2(\Rn))\cap L^2([0,1], H^1(\Rn))$ and from \cite[Lemma 5]{ekpv08b} with $A+iB=1$
\begin{equation*}
\partial_t\widetilde u=\Delta\widetilde u +\widetilde V(x,t)\widetilde u,\ \text{in}\ \Rn\times (0,1],
\end{equation*}
with
\begin{equation*}
\widetilde V(x,t)=\tfrac{\alpha\beta}{(\alpha(1-t)+\beta t)^2}V(\tfrac{\sqrt{\alpha\beta}x}{\alpha(1-t)+\beta t},\tfrac{\beta t}{\alpha(1-t)+\beta t}).
\end{equation*}
Also, for $\gamma=\frac 1{2\delta}$
\begin{equation*}
\|e^{\gamma |x|^2}\widetilde u(0)\|=\|u(0)\|\ \text{and}\ \|e^{\gamma|x|^2}\widetilde u(1)\|=\|e^{|x|^2/\delta^2}u(1)\|. 
\end{equation*}
 From the log-convexity property of $\|e^{\gamma |x|^2}\widetilde u(t)\|$ established in \cite[Lemma 3]{ekpv08b}, we know that
\begin{equation}\label{E: 125}
\sup_{[0,1]}\|e^{\gamma |x|^2}\widetilde u(t)\|\le e^{N\left(1+\|\widetilde V\|_{L^\infty(\Rn\times [0,1])}^2\right)}\left(\|e^{\gamma |x|^2}\widetilde u(0)\|+\|e^{\gamma |x|^2}\widetilde u(1)\|\right).
\end{equation}
The last claim in \cite[Lemma 5]{ekpv08b} shows that with $s=\frac{\beta t}{\alpha(1-t)+\beta t}$,
\begin{equation}\label{E:126}
\|e^{\gamma |x|^2}\widetilde u(t)\|=\|e^{\left[\tfrac{\gamma\alpha\beta}{(\alpha s+\beta(1-s))^2}+\tfrac{\alpha-\beta}{4(\alpha s+\beta(1-s))}\right]|y|^2}u(s)\|,\ \text{for}\ 0\le t\le 1.
\end{equation}
From \eqref{E: 125} and \eqref{E:126}, we find that
\begin{equation}\label{E:127}
\sup_{[0,1]}\|e^{\frac{t|x|^2}{\left(\delta+2-2t\right)^2}}u(t)\|\le e^{N\left(1+\|V\|_{\infty}^2\right)}\left[\|u(0)\|+\|e^{ |x|^2/\delta^2} u(1)\|\right].
\end{equation}
We then begin an inductive procedure where at the $k$th step we have constructed $k$ smooth functions, $a_j:[0,1]\longrightarrow [0,+\infty)$ verifying
\begin{equation}\label{E: cadena}
0<a_1<a_2<\dots<a_k<\dots\le \frac 1{\delta^2-4}\, ,\ \text{in}\ (0,1),
\end{equation}
\begin{equation}\label{E: algunas propiedades mas}
a_j(0)= 0,\ a_j(1)=1/\delta^2,\ \left(e^{8A_j}a_j\right)''> 0,\ \text{in}\  [0,1],
\end{equation}
\begin{equation}\label{E: algoagradable222}
\sup_{[0,1]}\|e^{a_j(t) |x|^2}u(t)\|\le  e^{N\left(1+\|V\|_{\infty}^2\right)}\left[\|u(0)\|+\|e^{ |x|^2/\delta^2} u(1)\|\right],
\end{equation}
when $j=1,\dots,k$, with $A_j'=a_j$, $A_j(1)=0$.
The case $k=1$ follows from \eqref{E:127} with $a_1(t)=t/\left(\delta+2-2t\right)^2$. Assume now that $a_1,\dots ,a_k$ have been constructed and let $b_k$ and $T_k$ be the functions defined in Lemma \ref{L: el calculo del conmutador mas largo2} for $a=a_k$. Then,
\begin{multline}\label{E: quealegria}
\|e^{a_k(t)|x|^2+b_k(t)x\cdot\xi-T_k(t)|\xi|^2}u(t)\|^2\\\le e^{2N\left(1+\|V\|_{\infty}^2\right)}\left(\|u(0)\|+\|e^{|x|^2/\delta^2}u(1)\|\right)^2,
\end{multline}
for $0\le t\le 1$ and all $\xi\in\Rn$. Observe that \eqref{E: quealegria} and the existence of the solutions $u_R$ defined in \eqref{E: el enemigo} imply that  $T_k>0$ in $(0,1)$, when $\delta>2$. Otherwise, \eqref{E: quealegria} implies that $u_R\equiv 0$, when $2\sqrt{1+R^2}<\delta$.

For $\e>0$, multiply \eqref{E: quealegria} by $e^{-2\e T_k(t)|\xi|^2}$ and integrate the new inequality with respect to $\xi$ in $\Rn$. It gives,
\begin{equation*}
\sup_{[0,1]}\|e^{a_{k+1}^\e(t)|x|^2}u(t)\|\le \left(1+\tfrac{1}{\e}\right)^{\frac n4}e^{N\left(1+\|V\|_{\infty}^2\right)}\left(\|u(0)\|+\|e^{|x|^2/\delta^2}u(1)\|\right),
\end{equation*}
with
\begin{equation*}
a_{k+1}^\e =a_k+\frac{b_k^2}{4\left(1+\e\right)T_k}\, .
\end{equation*}
On the other hand, $e^{8A_k}b_k$ is strictly convex and $b_k<0$ in $[0,1]$, 
\begin{equation}\label{E: unaecuacioncilla}
b_k(t)=2\left(a_k(t)-te^{-8A_k(t)}\delta^{-2}\right)
\end{equation}
and
\begin{equation*}
T_k(t)=2\int_0^tb_k^2(s)\,ds-a_k(t)-8\int_0^ta_k^2(s)\,ds-\alpha_k\int_0^te^{-8A_k(s)}\,ds,
\end{equation*} 
with 
\begin{equation*}
\alpha_k=\left(2\int_0^1b_k^2(s)\,ds-\frac 1{\delta^2}-8\int_0^1a_k^2(s)\,ds\right)\left(\int_0^1e^{-8A_k(s)}\,ds\right)^{-1}.
\end{equation*}
The last two formulae and \eqref{E: cadena} show that there is $N_\delta \ge 1$, independent of $k\ge 1$, such that 
\begin{equation}\label{E: elegir}
T_k(t)\le 2\left(\int_0^tb^2_k(s)\,ds+N_\delta\right)\ \text{and}\ N_\delta+\frac{b_k}2\ge 1,\ \text{in}\ [0,1]. 
\end{equation}
Also, $\left(\left(a'_k+4a_k^2\right)e^{16A_k}\right)'=e^{8A_k}\left(e^{8A_k}a_k\right)''$, $\left(a_k'+4a_k^2\right)e^{16A_k}$ is non decreasing in $[0,1]$ and 
\begin{equation}\label{E:monotonia}
a'_k+4a_k^2\ge 0\ \text{in}\  [0,1].
\end{equation}
 Set then,
\begin{equation}\label{E: metodoconstrucion}
a_{k+1}(t)=a_k(t)+\frac{b_k^2(t)}{8\left(\int_0^tb_k^2(s)\,ds+N_\delta\right)}\, .
\end{equation}
We have, $a_k <a_{k+1}$ in $(0,1)$, $a_{k+1}(0)=0$, $a_{k+1}(1)=\frac 1{\delta^2}$, 
\begin{equation}\label{E:comoodefino}
A_{k+1}=A_k+\frac 18 \log{\left(\int_0^tb_k^2(s)\,ds + N_\delta\right)}-\frac 18 \log{\left(\int_0^1b_k^2(s)\,ds + N_\delta\right)},
\end{equation}
and
\begin{equation*}
\sup_{[0,1]}\|e^{\left(a_{k+1}(t) -\e\right)|x|^2}u(t)\| < +\infty,\ \text{for all}\ \e>0.
\end{equation*}
The identity $\left(e^{8A}\right)'''=8\left(e^{8A}a\right)''$ and \eqref{E:comoodefino} show that $\left(e^{8A_{k+1}}a_{k+1}\right)''$ is a positive multiple of   
\begin{multline*}
\left(e^{8A_{k}}\left(\int_0^tb_k^2(s)\,ds+N_\delta\right)\right)'''=\left(e^{8A_{k}}\right)''' \left(\int_0^tb_k^2(s)\,ds+N_\delta\right)\\
+3\left(e^{8A_{k}}\right)''b_k^2+6\left(e^{8A_{k}}\right)'b_kb_k'+2e^{8A_{k}}\left(b_k''b_k+b_k'^2\right)
\end{multline*}
The equation verified by $b_k$ shows that the last sum is equal to
\begin{multline*}
\left(e^{8A_{k}}\right)''' \left(\int_0^tb_k^2(s)\,ds+N_\delta+\frac{b_k}2\right)
\\+8\left(a_k'+8a_k^2\right)e^{8A_{k}}b_k^2+2e^{8A_{k}}b_k'^2+16e^{8A_{k}}a_kb_kb_k'.
\end{multline*}
From \eqref{E: elegir} and \eqref{E:monotonia}, the above sum is bounded from below by
\begin{equation*}
\left(e^{8A_{k}}\right)''' +2e^{8A_{k}}\left(4a_kb_k+b_k'\right)^2>0,\ \text{in}\ [0,1].
\end{equation*}
The later and Lemma \ref{L: el calculo del conmutador mas largo2} show that \eqref{E: algoagradable222} holds up to  $j=k+1$. Finally, because \eqref{E:monotonia} holds with $k$ replaced by $k+1$,
\begin{equation*}
-\left(\frac 1{a_{k+1}}\right)'+4\ge 0,\ \text{in}\ (0,1],
\end{equation*}
 and the integration of this identity over $[t,1]$ shows that $a_{k+1}(t)\le \frac 1{\delta^2-4}$ in $(0,1)$. 
 
 Thus, there exists $a(t)=\lim_{k\to +\infty}a_k(t)$ and from \eqref{E: metodoconstrucion}, $\lim_{k\to +\infty}b_k(t)=0$. This and \eqref{E: unaecuacioncilla} show that 
 \begin{equation}\label{E: a integrar}
 ae^{8A}=t\delta^{-2}, \text{in}\ [0,1].
 \end{equation}
 Write $a(1)= 1/\delta^2$ as $1/4\left(1+R^2\right)$, for some $R>0$. Then, $a(t) = t/4\left(t^2+R^2\right)$ follows from the integration of \eqref{E: a integrar} and \eqref{E: loque hay que conseguir} from \eqref{E: algoagradable222} after letting $j\to +\infty$. Finally, when $\delta= 2$, we have
 \begin{equation*}
\sup_{[0,1]}\|e^{t|x|^2/4\left(t^2+R^2\right)}u(t)\| \le e^{N\left(1+\|V\|^2_{\infty}\right)} \left[\|u(0)\|_{L^2(\Rn)}+\|e^{|x|^2/4}u(1)\|\right],
\end{equation*}
for all $R>0$. Letting $R\to 0^+$, we get
 \begin{equation*}
\sup_{[0,1]}\|e^{|x|^2/4t}u(t)\| \le e^{N\left(1+\|V\|^2_{\infty}\right)} \left[\|u(0)\|_{L^2(\Rn)}+\|e^{|x|^2/4}u(1)\|\right],
\end{equation*}
and it implies, $u\equiv 0$.
\end{proof}
\begin{remark}\label{R:1} Theorem \ref{T: lamejora1} holds when \eqref{E: 0} and \eqref{E: loque hay que conseguir} are replaced respectively by
\begin{equation*}
\|e^{Tx_1^2/4(T^2+R^2)}u(T)\|_{L^2(\Rn)}<+\infty
\end{equation*}
and
\begin{multline*}
\sup_{[0,T]}\|e^{tx_1^2/4(t^2+R^2)}u(t)\|_{L^2(\Rn)}\\
\le e^{N\left(1+T^2\|V\|^2_{L^\infty(\Rn\times [0,T])}\right)}\left[\|u(0)\|_{L^2(\Rn)}+\|e^{Tx_1^2/4(T^2+R^2)}u(T)\|_{L^2(\Rn)}\right].
\end{multline*}

\end{remark}
\end{section}
%%%%%%%%%%%%%%%%%%%%%%%%%%%%%%%%%%%%%%%%%%%%%%%%%%%%%%%%%%%%%%%%%%%%%%%%%%%%%%%%%%%%%%%%%%%%%%%%%%%%%%%%%%%%%%%%%%%%%%%%%%%

%%%%%%%%%%%%%%%
%%%%%%%%%%%%%%%%%%%%%%%%%%

\begin{thebibliography}{99}

\bibitem{bonamie}  A. Bonami, B. Demange, \emph{A survey on uncertainty principles related to quadratic forms}. Collect. Math. Vol. Extra (2006) 1--36. 

\bibitem{bonamie2}  A. Bonami, B. Demange, P. Jaming, \emph{Hermite functions and uncertainty principles for the Fourier and the windowed Fourier transforms}. Rev. Mat. Iberoamericana {\bf 19},1  (2006) 23--55.

%\bibitem{BoKe05} J. Bourgain, C. E. Kenig, \emph{On  localization in the continuous Anderson-Bernoulli model in higher dimensions,} Invent. Math. {\bf 161}, 2 (2005) 1432--1297.

\bibitem{CoPr} M. Cowling, J. F. Price, \emph{Generalizations of Heisenberg's inequality}, Harmonic Analysis (Cortona, 1982) Lecture Notes in Math.,{\bf 992} (1983), 443-449, Springer, Berlin.

\bibitem{CoPr2} M. Cowling, J. F. Price. \emph{Bandwidth versus time concentration: the Heisenberg--Pauli--Weyl inequality}.
SIAM J. Math. Anal. {\bf 15} (1984) 151--165.

\bibitem{ds} H. Dong, W. Staubach.  \emph{Unique continuation for the Schr\"odinger equation with gradient vector potentials}. Proc. Amer. Math. Soc.
 {\bf 135}, 7 (2007) 2141--2149.

\bibitem{ekpv06} L. Escauriaza, C.E. Kenig, G. Ponce, L. Vega,  \emph{On Uniqueness  Properties of Solutions of Schr\"odinger  Equations}. Comm. PDE. {\bf 31}, 12 (2006) 1811--1823.

%\bibitem{ekpv07} L. Escauriaza, C.E. Kenig, G. Ponce, L. Vega,  \emph{On  Uniqueness Properties of Solutions of the k-generalized KdV,} J. of Funct. Anal. {\bf 244}, 2 (2007) 504--535.

\bibitem{ekpv06a} L. Escauriaza, C.E. Kenig, G. Ponce, L. Vega, \emph{Decay at Infinity of Caloric Functions within Characteristic Hyperplanes}, Math. Res. Letters {\bf 13}, 3 (2006) 441--453.

\bibitem{ekpv08a} L. Escauriaza, C.E. Kenig, G. Ponce, L. Vega,  \emph{Convexity of Free Solutions of Schr\"odinger Equations with Gaussian Decay}. Math. Res. Lett. {\bf 15}, 5 (2008) 957--971.

\bibitem{ekpv08b} L. Escauriaza, C.E. Kenig, G. Ponce, L. Vega,  \emph{Hardy's Uncertainty Principle, Convexity and Schr\"odinger Evolutions}. J. Eur. Math. Soc. {\bf 10}, 4 (2008) 883--907.

\bibitem{ekpv09} L. Escauriaza, C.E. Kenig, G. Ponce, L. Vega.  \emph{The Sharp Hardy Uncertainty Principle for Schr\"odinger Evolutions}. Duke Math. J. {\bf 155}, 1 (2010) 163--187.

\bibitem{cekpv10} M. Cowling, L. Escauriaza, C.E. Kenig, G. Ponce, L. Vega.  \emph{The Hardy Uncertainty Principle Revisited}. Indiana U. Math. J. {\bf 59}, 6 (2010) 2007--2026.

\bibitem{ekpv09b} L. Escauriaza, C.E. Kenig, G. Ponce, L. Vega.  \emph{Uncertainty Principle of Morgan Type for Schršdinger Evolutions}. J. London Math. Soc. {\bf 83}, 1 (2011) 187--207.

\bibitem{ekpv12} L. Escauriaza, C.E. Kenig, G. Ponce, L. Vega.  \emph{Uniqueness Properties of Solutions to Schr\"odinger Equations}. Bull. (New Series) of the Amer. Math. Soc. {\bf 49} (2012) 415--442.

%\bibitem{e98} L.C. Evans,  \emph{Partial Differential Equations.} Amer. Math. Soc. (1998)

%\bibitem{Hor1} L.  H\"ormander, \emph{Linear partial differential operators}, Berlin, Springer (1969).

\bibitem{Hardy} G.H. Hardy, \emph{A Theorem Concerning Fourier Transforms}, J. London Math. Soc. s1-8 (1933) 227--231.

\bibitem{Hor2} L. H\"ormander, \emph{A uniqueness theorem of Beurling for Fourier transform pairs}, Ark. Mat. {\bf 29}, 2 (1991) 237--240.

\bibitem{Ioke04} A. D. Ionescu, C. E. Kenig, \emph{$L^p$-Carleman inequalities and uniqueness of solutions of nonlinear Schr\"odinger equations,} Acta Math. {\bf 193}, 2 (2004) 193--239.

\bibitem{Ioke06} A. D. Ionescu, C. E. Kenig, \emph{Uniqueness properties of solutions of Schr\"odinger equations,} J. Funct. Anal. {\bf 232} (2006) 90--136.

\bibitem{kpv02} C.E. Kenig, G. Ponce, L. Vega, \emph{On unique continuation for nonlinear Schr\"odinger equations}, Comm. Pure Appl. Math. {\bf 60} (2002) 1247--1262.

\bibitem{kpv03} C.E. Kenig, G. Ponce, L. Vega. \emph{A Theorem of Paley-Wiener Type for Schr\"odinger Evolutions}. Annales Scientifiques  Ec. Norm. Sup. {\bf 47} (2014) 539-557.

\bibitem{SiSu} A. Sitaram, M. Sundari, S. Thangavelu, {Uncertainty principles on certain Lie groups},  Proc. Indian Acad. Sci. Math. Sci. {\bf 105} (1995), 135-151

%\bibitem{kpv03}  C.E. Kenig, G. Ponce, L. Vega, \emph{On unique continuation of solutions to the generalized KdV equation},  Math. Res. Letters {\bf10} (2003) 833--846.

\bibitem{StSh} E.M. Stein, R. Shakarchi, \emph{Princeton Lecture in Analysis II. Complex Analysis,} Princeton University Press (2003).

\bibitem{Temam} R. Temam, \emph{Navier-Stokes Equations, Theory and Numerical Analysis.} Amer. Math. Soc (1977). 

\end{thebibliography}
\end{document}